\documentclass[12pt,reqno]{amsart}
\usepackage{amsmath,bm}
\usepackage{amsfonts}
\usepackage{amssymb}
\usepackage{amsthm}
\usepackage{amscd}
\usepackage{a4wide}

\usepackage{amsmath}
\usepackage{amssymb}
\usepackage{amsthm}
\usepackage{graphics}
\usepackage{eucal}
\usepackage{mathrsfs}
\usepackage{color}

\usepackage{color}
\theoremstyle{plain}
\newtheorem{theorem}{Theorem}[section]
\newtheorem{proposition}[theorem]{Proposition}
\newtheorem{lemma}[theorem]{Lemma}
\newtheorem{corollary}[theorem]{Corollary}

\theoremstyle{definition}

\newcommand{\C}{{\mathbb{C}}}

\newcommand{\R}{{\mathbb{R}}}
\newcommand{\Z}{{\mathbb{Z}}}
\newcommand{\N}{{\mathbb{N}}}
\newcommand{\Q}{{\mathbb{Q}}}
\newcommand{\br}{\mathbf{r}}

\renewcommand{\O}{{\mathcal{O}}}

\newcommand{\height}{\operatorname{H}}

\newcommand{\e}{{\varepsilon}}

\newcommand{\SL}{\operatorname{SL}}

\newcommand{\Mat}{\operatorname{Mat}}

\newcommand{\bad}{\operatorname{Bad}}
\newcommand{\Res}{\operatorname{Res}}
\newcommand{\Image}{\operatorname{Im}}
\newcommand{\Id}{\operatorname{Id}}

\begin{document}
\title{Badly approximable vectors, $C^{1}$ curves and number fields}

\begin{abstract}
We show that the set of points on $C^{1}$ curves which are badly approximable by rationals in a number field form a winning set in the sense of W. Schmidt. As a consequence, we obtain a number field version of Schmidt's conjecture in Diophantine approximation. 
\end{abstract}


\author{Manfred Einsiedler}
\address{ETH Z\"{u}rich, Departement Mathematik, R\"{a}mistrasse 101 8092, Z\"{u}rich Switzerland}\email{manfred.einsiedler@math.ethz.ch, beverly.lytle@math.ethz.ch}
\author{Anish Ghosh}
\thanks {This work was partly supported by a Royal Society International Grant. The first
and last named authors acknowledge the support of SNF-grant 200021-127145.}
\address{School of Mathematics, University of East Anglia, Norwich, NR4 7TJ UK}\email{ghosh.anish@gmail.com}
\author{Beverly Lytle}

\maketitle
\tableofcontents

\section{Introduction}

Recall that a real number $x$ is \emph{badly approximable} if there exists $c > 0$ such that
\begin{equation}\label{eq:bad1}
|qx - p| > \frac{c}{q}
\end{equation}
\noindent for all $q \in \N$ and $p \in \Z$. It is well known that badly approximable vectors have zero Lebesgue measure and full Hausdorff dimension (Jarnik \cite{J} for $n=1$ and Schmidt \cite{S1, S3} for arbitrary $n$). In fact, Schmidt showed that they are \emph{winning} for a certain game, a stronger and more versatile property than having full Hausdorff dimension. 

\subsection{Schmidt's Game}

In \cite{S1}, Schmidt introduced the following game.  Two players, say Player A and Player B, start with a complete metric space $X$, a subset $W\subseteq X$, and two parameters $0<\alpha, \beta<1$.  Player A begins by choosing an arbitrary ball $A_0=B(x_0,\rho)$.  The Player B then chooses a ball $B_0=B(y_1,\alpha\rho)$ contained in $A_0$.  Player A makes his next move by choosing a ball $A_1 \subset B_0$ of radius $\alpha\beta\rho$.  The $n$th step of the game consists of first the Player A choosing a ball $A_n = B(x_n, (\alpha\beta)^n\rho) \subset B_{n-1}$ and Player B following by choosing the next ball $B_n=B(y_n,\alpha(\alpha\beta)^n\rho) \subset A_n$.  As the radii of the balls are shrinking to zero and $X$ is complete, at the end of the infinite game, Player A and Player B are left with a single point $\{x_{\infty}\} =\bigcap_n A_n$.  We say that Player B has won this $(\alpha,\beta)$ game if $x_{\infty} \in W$. The set $W$ is called $(\alpha, \beta)$-winning if Player B can find a winning strategy, $\alpha$-winning if it is $(\alpha, \beta)$-winning for all $0 < \beta < 1$ and winning if it is $\alpha$-winning for some $\alpha > 0$. Schmidt games have the following properties (cf. \cite{S1}, \cite{Dani3}):
\begin{enumerate}
\item A winning subset of $\R^n$ is \emph{thick}, i.e.\ the intersection of a winning set with every open set in $\R^n$ has Hausdorff dimension $n$.\\
\item A countable intersection of $\alpha$-winning sets is $\alpha$-winning.\\
\item Winning sets are preserved by bi-Lipschitz homeomorphisms of $\R^n$.\\
\item The set of badly approximable vectors is $(\alpha, \beta)$-winning whenever $2\alpha < 1 + \alpha\beta$; in particular, it is $\alpha$-winning for any $0 < \alpha \leq 1/2$.
\end{enumerate}

\subsection{Diophantine approximation in number fields} 
 
Let $K$ be a number field of degree $d$ with $r$ real and $s$ complex embeddings.  Denote by $S$ the set of Galois embeddings $\sigma$, where for the complex embeddings, one chooses one of the pair $\sigma$ and $\bar{\sigma}$.   Let $\O_K$ be the ring of integers of $K$.  Denote by $K_S:=\R^r\times\C^s \cong \R^d$.  We denote by $\tau$ the twisted diagonal embedding of $K$ into $K_S$ by
$$\tau(x) = (\sigma(x))_{\sigma\in S},$$
\noindent where we identify each coordinate of $K_S$ with an element of $S$. The notation will be extended to vectors and matrices and~$\tau$ will be omitted in the notation when it causes no confusion.\\

It is natural to ask if analogues of the traditional theorems in Diophantine approximation hold in the setting of number fields. More precisely, we wish to approximate elements in $K_S$ using ratios of elements in $\O_K$. Analogues of Dirichlet's theorem in this setting have been established by several authors (cf. \cite{S4}, \cite{B}, \cite{Q}, \cite{H}) using appropriate adaptations of the geometry of numbers. Moreover, \cite{B} and \cite{H} also show the existence of badly approximable vectors\footnote{We note, however, that our notion of badly approximable vectors differs slightly
from the notion considered elsewhere as we do not square the absolute value at the complex places in our definition.} in this setting.\\ 

Say that a vector ${\bf x} = (x^{\sigma})_{\sigma\in S} \in K_S$ is $K$-badly approximable (${\bf x} \in \bad(K)$) if there exists $c>0$ such that for all $p, q\in \O_K$ with $q\neq 0$ 
\begin{equation}\label{eq:bad}
\max_{\sigma\in S}\{|\sigma(p)+x^{\sigma}\sigma(q)|\}\max_{\sigma\in S}\{|\sigma(q)|\} > c.
\end{equation}

\noindent Here and in the rest of the paper, $|~|$ will be used to denote both real and complex absolute values depending on context. 

S.G. Dani \cite{Dani1} showed that a real number is badly approximable if and only if a lattice associated to the number has a bounded trajectory in the space $\SL_2(\Z) \backslash \SL_2(\R)$ under the action of a certain subsemigroup.
A version of the Dani correspondence (\S \ref{Dani}) states that a vector ${\bf x}$ is $K$-badly approximable if and only if the trajectory 
\begin{equation}\nonumber
\left\{ \SL_2(\O_K) \left(\begin{pmatrix}1& x^{\sigma} \\ &1  \end{pmatrix}\right)_{\sigma\in S}  g_t \,|\, t\geq 0  \right\}
\end{equation}
\noindent for the flow
\begin{equation}\label{eq:diag}
g_{t} := \left\{ \left(\begin{pmatrix} e^{-t}& \\ &e^{t} \end{pmatrix} \right)_{\sigma \in S} : t \geq 0 \right\}
\end{equation}
\noindent is bounded in the quotient space $\SL_2(\O_K) \backslash \SL_2(K_S)$. In conjunction with the Moore ergodicity theorem, we can conclude that $K$-badly approximable vectors have zero Lebesgue measure. Nevertheless, they constitute a winning set for Schmidt's game and therefore have full Hausdorff dimension.     

\subsection{Main Results}
 
\noindent We show that the set of badly approximable vectors are winning even when the game is restricted to a curve.  We need slightly separate conditions on the curve in question in different cases. We recall that $ r+s$ is the number of
simple factors of~$\SL_2(K_S)$.  

\begin{theorem}\label{thm:main1}
Let $\phi=(\phi_{\sigma})_{\sigma \in S}:[0,1] \to K_S$ be a continuously differentiable map. We assume that  $\phi'_{\sigma}(x) \neq 0$ for all but finitely many $x \in [0,1]$ and for all $\sigma$ in a subset $S' \subset S$ (possibly depending on~$x$) with 
$$
 |\{\sigma\in S':\sigma\textrm{ is real}\}|+ 2|\{\sigma\in S':\sigma\textrm{ is complex}\}|   > \lfloor \frac{d}{2}\rfloor.
$$
\noindent Define
\begin{equation}\label{eq:Phi}
\Phi(x):=\left(\begin{pmatrix} 1&\phi_{\sigma}(x) \\ &1 \end{pmatrix} \right)_{\sigma\in S},
\end{equation} and let $g_t$ as defined in (\ref{eq:diag}) act on $\SL_2(\O_K) \backslash \SL_2(K_S)$ by right multiplication.  
Let $\Lambda= \SL_2(\O_K)g\in \SL_2(\O_K) \backslash \SL_2(K_S)$.  
Then the set
\begin{equation}
\{x\in[0,1]\, |\,  \Lambda\Phi(x)\text{ has bounded trajectory under the flow } g_t\}
\end{equation}
is winning in the sense of Schmidt, and hence has Hausdorff dimension $1$.  
\end{theorem}

We note that the condition on the curve in the theorem above is not of a technical nature. If~$\Phi$ simply
parametrizes a line segment that is parallel a coordinate axis (which are special directions as they correspond
to the simple factors of the ambient Lie group), then there may not be any points with bounded trajectory on the 
line segment if e.g.~$d=r=3$, see~\S \ref{counterex}. 

Theorem \ref{thm:main1} coupled with Dani's correspondence gives us:
\begin{corollary}\label{cor:main1}
Let $\phi: [0,1] \to K_S$ be as in Theorem \ref{thm:main1}. Then the set 
\begin{equation}\nonumber
\{x\in [0,1]\,|\, \phi(x) \in \bad(K)\}
\end{equation}
is winning in the sense of Schmidt, and hence has Hausdorff dimension $1$.  
\end{corollary}

 Now we let $K$ be a real quadratic extension of $\Q$ where~$d=r=2$. In this case, we can choose
different directions in the two-dimensional Cartan subgroup. 

\begin{theorem}\label{thm:main2}
	Let~$K$ be a real quadratic field. 
Let $\phi=(\phi_{\sigma})_{\sigma \in S}:[0,1] \to K_S$ be a continuously differentiable map such that $\phi'_{\sigma}(x) \neq 0$ for all but finitely many $x \in [0,1]$ and every $\sigma \in S$. Let ${\bf r} \in \R^2$ be a real vector with $r_{\sigma} \geq 0$ for $\sigma \in S$ and $\sum_{\sigma}r_{\sigma} = 1$.  For $t \geq 0$, let 

\begin{equation}\label{eq:defdiag}
g(\br)_{t} := \left( \begin{pmatrix} e^{-r_{\sigma} t}& \\ &e^{r_{\sigma} t} \end{pmatrix} \right)_{\sigma \in S}
\end{equation}
act on $\SL_2(\O_K) \backslash \SL_2(K_S)$ by right multiplication.  Let $\Lambda= \SL_2(\O_K)g\in \SL_2(\O_K) \backslash \SL_2(K_S)$.  
Then the set
\begin{equation}
\{x\in[0,1]\, |\,  \Lambda\Phi(x)\text{ has bounded trajectory under the flow } g(\br)_t\}
\end{equation}
is winning in the sense of Schmidt, and hence has Hausdorff dimension $1$.  
\end{theorem}

\noindent The $\alpha$ in the winning statement above does not depend on $\br$. Taking intersections over rational vectors $\br$ and using the fact that countable intersections of $\alpha$-winning sets are $\alpha$-winning therefore shows us that

\begin{corollary}\label{cor:main2}
With notation as in Theorem \ref{thm:main2}, 
\begin{equation}
\bigcap_{\br \in \Q^{2}_{+}}\left\{x\in[0,1]\, |\,  \Lambda\Phi(x)\text{ has bounded trajectory under the flow } g(\br)_t\right\}
\end{equation}
\noindent is winning.
\end{corollary}

\noindent Remarks:
\begin{enumerate}
\item Corollary \ref{cor:main2} provides an analogue of Schmidt's conjecture for real quadratic fields, a theorem of Badziahin-Pollington-Velani \cite{BPV} in the real case. See also the works \cite{An1, An2} of Jinpeng An for stronger results in this vein.
\item In contrast with Corollary \ref{cor:main2}, we note that it follows from results in \cite{EKL} (see also \cite{EK}) that bounded orbits for the full Cartan subgroup have \emph{zero} Hausdorff dimension.
\item  Using Theorem \ref{thm:main1} and the Marstrand Slicing Theorem, we see that $\bad(K)$ has full Hausdorff dimension.
\item As far as we are aware, the only other result regarding abundance of badly approximable vectors in the context of number fields is for certain quadratic extensions whose rings of integers have unique factorization \cite{EKr}.
\item Theorems \ref{thm:main1}, \ref{thm:main2} and Corollary \ref{cor:main1} are among the very few existing results which show that badly approximable points on curves are winning. In a recent work, V. Beresnevich \cite{Ber} show that badly approximable vectors on ``nondegenerate" manifolds have full Hausdorff dimension.  See also the works \cite{KW1, KW2, KTV} for results regarding badly approximable vectors on certain classes of fractals. 
\item We refer to the $r_i$'s which appear in $g(\br)_t$ as weights. Thus Theorem \ref{thm:main1} deals with equal weights and Theorem \ref{thm:main2} with unequal weights. The equal weights version of our results is closely related to a result from \cite{BFKRW} (see Proposition $4.9$). We note, however, that in the context of this paper certain special directions (e.g. line segments parallel to a coordinate axis corresponding to one of three real factors) may fail to have any badly approximable points on them while line segments in general directions are covered by Theorem~\ref{thm:main1}.
\item   In a related, earlier result \cite{Aravinda}, it is shown that points on $C^1$ curves in rank $1$ locally symmetric spaces which have bounded orbits under the geodesic flow are winning. The result of this paper may be viewed
as a generalization of this result to certain quotients of higher rank groups with~$\Q$-rank 1.

\end{enumerate}

\subsection*{Acknowledgements} AG thanks ETH Z\"urich for hospitality during visits.

\section{Mahler's Compactness Criterion}

In this section, we state and prove Mahler's compactness criterion for $S$-adic homogeneous spaces. Theorem \ref{thm:Mahler}, which is the main result in this section is almost certainly well known to experts. For instance, see \cite{KT1, KT2}. We provide a proof for completeness. The original statement of Mahler's compactness criterion is as follows:

\begin{theorem}
A subset $A \subset \SL_n(\Z)\backslash \SL_n(\R)$ is relatively compact if and only if there exists $\varepsilon >0$ such that for all $\SL_n(\Z)g \in A$ and for all $v \in \Z^n g$, $\Vert v \Vert >\varepsilon$. 
\end{theorem}

Here we take $\Vert v \Vert$ to be the maximum norm.  We wish to rephrase this theorem for the space $\SL_2(\O_K)\backslash \SL_2(K_S)$.  To do this we will use restriction of scalars to map $\SL_2(\O_K)\backslash \SL_2(K_S)$ in $\SL_{2d}(\Z)\backslash \SL_{2d}(\R)$.  More concretely, we define the embedding as follows:

Since the degree of $K$ over $\Q$ is $d$, we may view $K$ as a $d$ dimensional vector space over $\Q$.  We choose a basis $\{a_1,\dots a_d\}$ of $K$ over $\Q$ so that the $\Z$-span of these elements is $\O_K$.  Left multiplication by an element of $K$ is a $\Q$-linear transformation on $K$.  We have an algebraic embedding $\iota:K\to \Mat_{d,d}(\Q)$ with respect to this chosen basis.  We then have an induced map
\begin{align*}
\iota:\SL_2(K) &\to \SL_{2d}(\R) \\
\begin{pmatrix}a&b\\c&d\end{pmatrix} &\mapsto \begin{pmatrix} \iota(a)&\iota(b)\\ \iota(c)&\iota(d)\end{pmatrix}
\end{align*}

Thus, we have defined the algebraic subgroup
$$\Res_{K/\Q}\SL_2= \left\{\begin{pmatrix} A&B\\C&D\end{pmatrix} | A,B,C,D\in \Image(\iota), AD-BC = \Id \right\},$$
such that $\iota(\SL_2(K)) = (\Res_{K/\Q}\SL_2) (\Q) \subseteq \SL_{2d}(\Q).$

To give another idea of the structure of this space, consider the basis of $K$ given by $\{1,\xi,\xi^2,\dots,\xi^{d-1}\}$, where $\xi$ is a primitive element of $K$, that is, $K=\Q(\xi)$.  (Note that the transformation from this basis to the previously mentioned one is rational.)  Then as a $\Q$-linear transformation of $K$, left multiplication by $\xi$ is represented by the companion matrix
\[
T_{\xi}=
\begin{pmatrix} 0&1&0&\cdots&0 \\ 
0&0&1&\cdots&0 \\
\vdots &\vdots &\vdots & \ddots &\vdots \\
-c_0&-c_1&-c_2&\cdots&-c_{d-1} \end{pmatrix}
,
\]
where the $p_{\min}(x) = x^d + c_{d-1}x^{d-1}+ \cdots + c_1 x +c_0$ is the minimal polynomial of $\xi$ over $\Q$.  It is well known that this is conjugated by the Vandermonde matrix  $V_{\xi}$
(associated to the various Galois embeddings of $\xi$) to the diagonal form
\[
V_{\xi}^{-1}T_{\xi}V_{\xi} = \text{diag}(\sigma(\xi))_{\sigma\in S}
\]
Since the elements of $\Image(\iota)$ commute, they are simultaneously diagonalizable.  Moreover for any $\zeta \in K$ there exists $p_{\zeta}(x) \in \Q[x]$ of degree less than $d-1$ with $\zeta = p_{\zeta}(\xi)$, and so the transformation of left multiplication by $\zeta$ is given by $T_{\zeta}=p_{\zeta}(T_{\xi})$ and is conjugate to 
\[
\text{diag}(p_{\zeta}(\sigma(\xi)))_{\sigma\in S} = \text{diag}(\sigma(\zeta))_{\sigma\in S}
\]
\noindent (since the maps $\sigma$ are ring homomorphisms).  After a simple change of bases, we have an embedding $\SL_2(K) \to \SL_{2d}(\R)$ given by
\[
\begin{pmatrix}\zeta_1 &\zeta_2 \\ \zeta_3 & \zeta_4\end{pmatrix} \mapsto
\text{diag}\left(\begin{pmatrix}\sigma(\zeta_1) &\sigma(\zeta_2) \\ \sigma(\zeta_3) & \sigma(\zeta_4)\end{pmatrix}\right)_{\sigma \in S}.
\]
This is precisely $\iota(\SL_2(K))\subset \SL_2(K_S)$, where the latter is sitting in $\SL_{2d}(\R)$ in block diagonal form and $\iota(\SL_2(K))$ forms the $\Q$-points of the variety defined above after conjugation with the appropriate change of basis matrix.   

  As with the identification of $\SL_2(\Z)\backslash\SL_2(\R)$ with the space of covolume 1 lattices in $\R^2$, we have an identification of $\SL_2(\O_K)\backslash\SL_2(K_S)$ with the set, denoted $X$, of discrete (as subsets of $K_S^2$) rank 2 $\O_K$-modules with the property that for each $\Lambda \in X$ there exists a basis $\{v,w\}$ of $\Lambda$ so that for each $\sigma$, $\sigma(v)$ and $\sigma(w)$ form the sides of a parallelepiped of area 1 (in $\R^2$ or $\C^2$, appropriately). Now that we have a proper embedding $\SL_2(\O_K)\backslash\SL_2(K_S) \to \SL_{2d}(\Z)\backslash\SL_{2d}(\R)$, we use Mahler's compactness criterion on the second space to derive the 
statement:

\begin{theorem}\label{thm:Mahler}
A subset $A\subset X$ is relatively compact if and only if there exists $\varepsilon>0$ such that for all $\Lambda\in A$ and for all vectors $v\in \Lambda=\tau(\O_K^2)g$, $\Vert v \Vert > \varepsilon$.  
\end{theorem}

For a vector $v$ in $K_S^2$, denote by $v^{\sigma}$ the projection of $v$ onto the factor associated with the embedding $\sigma$.  We define a height function $\height : K_S^2 \to \R$ by
$$\height(v) := \prod_{\sigma}\Vert v^\sigma\Vert_\sigma=\prod_{\sigma\textrm{ real}} \Vert v^{\sigma} \Vert \prod_{\sigma\textrm{ complex}} \Vert v^{\sigma} \Vert^2,
$$
where we write~$\Vert \cdot\Vert_{\sigma}$ for the norm respectively the square of the norm depending on whether
the place~$\sigma$ is real or complex.
 It will be useful to think of this height function as a measure of depth into the cusp.  We wish to say that a set $A$ is relatively compact if and only if the height function is uniformly bounded below by a positive constant over all $\O_K$-modules in $A$.  To prove this statement, we first need some properties of the function $H$.

\begin{lemma}\label{lem:height1}
Let $\Lambda=\tau(\O_K)g\in X$ and $v \in \Lambda\setminus \{0\}$.  Then
\begin{enumerate}
\item $\height(v)\neq 0$
\item for $\xi \in \O_K^{\times}$, $\height(\xi v)= \height(v)$.
\end{enumerate}
\end{lemma}

\begin{proof}
For the first property, suppose $\height(v)=0$. Then $\Vert v^{\sigma} \Vert = \max\{|v_1^{\sigma}|,|v_2^{\sigma}|\}=0$ for some $\sigma$.  Since $v = \tau(a,b)g$ for $a,b\in \O_K$, and since $g$ is invertible, we have $(\sigma(a),\sigma(b)) = (0,0)$.  Thus, $(a,b) = (0,0)$ and hence $v=0$. The other property follows from the 
product formula~$\prod_{\sigma\textrm{ real}} |\sigma(\xi)|\prod_{\sigma\textrm{ complex}} |\sigma(\xi)|^2=1$ for units~$\xi\in\mathcal{O}^*$.  
\end{proof}

 The following lemma is essentially taken from the preprint \cite{KT1} of Kleinbock and Tomanov.

\begin{lemma}\label{lem:height2}
There exists a constant $C$ such that if $v\in K_S^2$ with $\height(v)\neq 0$ then there exists a unit $\xi$ with 
\begin{equation}\nonumber
C^{-1}\height(v)^{\frac{1}{d}} \leq \Vert \xi v \Vert \leq C\height(v)^{\frac{1}{d}}.
\end{equation}
\end{lemma}

\begin{proof}
Let 
\[
Z = \bigl\{{\bf x} = (x^{\sigma})_{\sigma\in S} \in \R^{r+s}_{>0} | \prod_{\sigma\textrm{ real}} x^{\sigma}
\prod_{\sigma\textrm{ complex}} (x^{\sigma})^2 = 1\bigr\}.
\]
The morphism $\xi \mapsto  (\vert \sigma(\xi)\vert)_{\sigma\in S}$ sends $\O_K^*$ to a subgroup of the multiplicative group $Z$.  By the proof of the Dirichlet Unit Theorem, this is a cocompact lattice.  Thus, there exists a constant $C$ so that for any $(x^{\sigma}) \in Z$ there exists $\xi \in \O_K^*$ with 
\[
C^{-1} \leq \vert \sigma(\xi) \vert x^{\sigma} \leq C.
\]

Let $v \in K_S^2$ with $\height(v)\neq 0$.  Then the vector $\left(\frac{\Vert v^{\sigma} \Vert}{\height(v)^{1\backslash d}}\right)$ is in $Z$.  Applying the previous lemma, we have the claim.
\end{proof}

\begin{proposition}
A subset $A\subset X$ is relatively compact if and only if there exists $\delta>0$ such that for all $\Lambda\in A$ and for all nonzero vectors $v\in \Lambda=\tau(\O_K^2)g$, $\height(v) > \delta$.  
\end{proposition}

\begin{proof}
The first implication follows from the  previous lemma along with Theorem  \ref{thm:Mahler}.  The reverse implication is immediate from continuity of~$H$.  
\end{proof}

\section{Dani's Correspondence}\label{Dani}

We prove a version of Dani's correspondence for number fields.  As in the introduction we will consider here
the notion of badly approximable vectors with equal weights, which as we will now show corresponds to the 
dynamics of the flow $g_t$.

\begin{proposition}\label{prop:dani}
A vector ${\bf x} \in K_S$ is $K$-badly approximable, that is, there exists $c>0$ such that for all $p,q \in \O_K$ with $q\neq 0$
\begin{equation}\label{eq:dani1}
\max_{\sigma\in S}\{\vert \sigma(q)x^{\sigma} + \sigma(p)\vert\}\max_{\sigma\in S}\{\vert \sigma(q)\vert\} > c,
\end{equation}
if and only if the trajectory
\begin{equation}\label{eq:dani2}
\left\{ \SL_2(\O_K) \left(\begin{pmatrix}1& x^{\sigma} \\ &1  \end{pmatrix}\right)_{\sigma\in S}  g_t \,|\, t\geq 0  \right\}
\end{equation}
 in the quotient space~$\SL_2(\O_K)\backslash\SL_2(K_S)$ is bounded.  
\end{proposition}

\begin{proof}
	By Mahler's compactness criterion, we know that boundedness of the trajectory is equivalent
	to the existence of positive~$c'$ such that
	\[
		 \max_{\sigma\in S}\bigl\|\tau(q,p)\left(\begin{pmatrix}1& x^{\sigma} \\ &1  \end{pmatrix}
		     \begin{pmatrix}e^{-t}&\\&e^t\end{pmatrix}\right)_{\sigma\in S} \bigr\|>c'
	\]
	or equivalently
	\[
	  \max_{\sigma\in S}\| (e^{-t}\sigma(q),e^t(\sigma(q)x^\sigma+\sigma(p)))\|>c'
	\]
	for all~$t\geq0$ and all nonzero pairs~$(q,p)\in\O_K^2$. 
	
	Assume first that the orbit is unbounded. Then there exists for every~$c'>0$ some~$t\geq 0$ and some
	nonzero vector~$(q,p)\in\O_K^2$ with
	\[
	 \max_{\sigma\in S}\|(e^{-t}\sigma(q),e^t(\sigma(q)x^\sigma+\sigma(p)))\|<c'.
	\]
	Note that since~$t\geq 0$, that~$q=0$ would contradict this inequality (at least for small enough~$c'$
	since~$p\in\O_K$ cannot be small at all places~$\sigma\in S$). Hence~$q\neq 0$.
	By splitting the above inequality into two inequalities for the first and second coordinates
	of the vectors involved and taking the product we obtain 
	\[
	\max_{\sigma \in S}\{\vert \sigma(q)x^{\sigma} + \sigma(p)\vert\}\max_{\sigma\in S}\{\vert \sigma(q)\vert\} < (c')^2.
	\]
	As~$c'$ was arbitrary we see that the vector~$(x^\sigma)_{\sigma\in S}$ is not badly approximable.
	
		Assume now that~${\bf x}=(x^\sigma)_{\sigma\in S}$ is not badly approximable. Then we have by definition 
	that for every~$c>0$ there exists~$p,q\in\O_K$
	with~$q\neq 0$ such that
	\[
	 \max_{\sigma\in S}\{|\sigma(p)+x^\sigma\sigma(q)|\}\max_{\sigma\in S}\{|\sigma(q)|\}<c.
	\]
	We choose~$t=-\frac12\log c+\log\max_{\sigma\in S}\{|\sigma(q)|\}$, which will be positive
	if only~$c$ is sufficiently small (as the second summand is bounded from below for~$q\in\O_K\setminus\{0\}$). 
  Note that this gives~$\max_{\sigma\in S}e^{-t}|\sigma(q)|=c^{\frac12}$. Dividing our assumed
  inequality by the latter equality we also get~$\max_{\sigma\in S}e^t|\sigma(p)+x^\sigma\sigma(q)|<c^{\frac12}$.
 As~$c$ was arbitrary, this shows that the orbit is not bounded.
\end{proof}

Notice that it is sufficient to have that the trajectory is bounded for a discrete sequence of times $t_n$ where the consecutive differences are uniformly bounded.

\section{A special case}
In this section, we give proofs of Theorems \ref{thm:main1} and \ref{thm:main2} in a simplified linear case. In the next section, we show that a modification of this simplified argument suffices for the general $C^1$ case as well.  
\subsection{A special case of Theorem \ref{thm:main1}}
We fix notation as in Theorem \ref{thm:main1}.  Fix $0<\alpha < \frac{1}{2}$ and let $0< \beta <1$.  Notice that we choose 
$\alpha$ is independent of $\beta$ as required by the definition of winning. Fix the base point $\SL_2(\O_K)g \in \SL_2(\O_K)\backslash \SL_2(K_S)$ and denote by $\Lambda = \tau(\O_K^2)g$ the associated discrete $\O_K$-module viewed as a subset of $\R^{2d}$. In the simplified case we suppose that the function $\Phi$ is such that $\phi_{\sigma}(x)=a_{\sigma}x$ with $a_\sigma\in\R$ and $a_{\sigma}\neq 0$ for sufficiently many $\sigma\in S$ as required in the theorem, to be precise at least one half of the factors should satisfy~$a_\sigma\neq 0$ with the complex places counting double. In this setting we will describe the strategy of Player B and prove that it is indeed winning.\\  

The ultimate goal of Player B is to have the point $x_{\infty}$ which remains at the end of the game satisfying that the set $\{v\in \Lambda\Phi(x_{\infty})g_{t}\setminus \{0\}~|~t \geq 0 \}$ is bounded away from zero.  As remarked before, it is sufficient to have that $\{v\in \Lambda\Phi(x_{\infty})g_{t_n}\setminus\{0\}~|~n \in \N \}$ is bounded away from zero where $t_n$ is a positive sequence tending to infinity with bounded gaps. During each round of the game, Player B will want to monitor the short vectors that will appear in the module $\Lambda \Phi(x_n)g_{t_n}$ and play in such a way that these vectors are expanding under $g_t$ after a finite amount of time (which needs to be independent of the short vector and the length of time the game has already been played).  The initial step of the game by Player A may force a vector to be short for a long time, and the initial step of Player B will only make sure that the perturbed vector grows at some point in the future.  However, for the later steps of the game it is important that we give a uniform lower bound on how short the perturbed vectors can become.\\

After the initial steps of the game, the $n$th round of the game plays out as follows:  Player A has chosen a subinterval $A_n = B(x_n, \rho(\alpha\beta)^n)\subseteq B_{n-1}$.  This corresponds to the collection of $\O_K$-modules 
$$\{\Lambda \Phi(x_n+x)g_{t_n})~|~x\in B(0,\rho(\alpha\beta)^n)\}$$
\noindent where $t_n=\frac{1}{2}\log\frac{1}{\rho(\alpha\beta)^n}$.  Player B focuses on the modules associated to the midpoint of $A_n$, namely $\Lambda \Phi(x_n)g_{t_n}$.  If this module contains no short vectors, i.e. no nonzero $v\in\Lambda\Phi(x_n)g_{t_n}$ with $\height(v)<1$, then Player B chooses the new ball $B_n$ heedlessly as allowed by the
rules of the game.  Suppose there does exist a nonzero $v\in\Lambda\Phi(x_n)g_{t_n}$ with $\height(v)<1$.  The strategy Player B employs in choosing the ball B makes use of the following phenomena\footnote{This is where we use
that our quotient has~$\Q$-rank one.}:

\begin{lemma}\label{lem:onespan}
Let $v= (v_1,v_2)$ and $w=(w_1,w_2)$ be two nonzero vectors of an $\O_K$-module $\Delta = \tau(\O_K^2)h \in X$.  Suppose $\height(v) \height(w) <1$.  Then $Kv=Kw$.  
\end{lemma}

\begin{proof}
We may write $v=(a_1,a_2)h$ and $w=(b_1,b_2)h$ with $a_1,a_2,b_1,b_2\in\O_K$.  Recall that $\Delta$ admits a basis as an $\O_K$-module whose projection to each factor corresponding to some $\sigma$ determines a parallelogram with area 1.  Consider, then, the parallelograms formed by the projections of $v$ and $w$ to each factor.  On one hand, the product of their areas (resp.\ the areas squared for the complex places) is given by
\begin{align*}
 \prod_{\sigma}\left|\det\begin{pmatrix}v^{\sigma}\\w^{\sigma}\end{pmatrix} \right|_\sigma &= 
\prod_{\sigma}\left|\det\begin{pmatrix}v_1^{\sigma}&v_2^{\sigma}\\w_1^{\sigma}&w_2^{\sigma}\end{pmatrix}\right|_\sigma 
\\ &\leq \prod_{\sigma}  
\max\{|v_1^{\sigma}|_\sigma,|v_2^{\sigma}|_\sigma\}\max\{|w_1^{\sigma}|_\sigma,|w_2^{\sigma}|_\sigma\} 
\\&= \height(v)\height(w)<1.
\end{align*}

On the other hand, using the fact that $h\in \SL_2(K_S)$ and so $\det(h^{\sigma})=1$ for all $\sigma$, we have
\[
\prod_{\sigma}\left|\det\begin{pmatrix}v^{\sigma}\\w^{\sigma}\end{pmatrix} \right|_\sigma = 
\prod_{\sigma}\left|\det\begin{pmatrix}a_1^{\sigma}&a_2^{\sigma}\\b_1^{\sigma}&b_2^{\sigma}\end{pmatrix}\right|_\sigma
 = \prod_{\sigma} |\sigma(a_1b_2-b_1a_2)|_\sigma=|N_{K|\Q}(a_1b_2-b_1a_2)|,
\]
which is an integer as $a_1b_2-b_1a_2\in \O_K$.
Therefore,
\[
\prod_{\sigma}\det\begin{pmatrix}v^{\sigma}\\w^{\sigma}\end{pmatrix} = 0,
\]
implying that~$a$ and~$b$ are~$K$-multiples of each other. As multiplication by elements of~$K$ on~$K_S^2$ 
commutes with~$h$, we see that also  $v$ and $w$ are $K$-linearly dependent.
\end{proof}

Thus in any round of the game Player B need only worry about a single $K$-span of short vectors.   Define $\Phi_{n}(x)=g_{t_n}^{-1}\Phi(x)g_{t_n}$ so that $\Lambda\Phi(x_n+x)g_{t_n} = \Lambda \Phi(x_n)g_{t_n}\Phi_{n}(x)$ and the vectors in~$\Lambda\Phi(x_n+x)g_{t_n}$ corresponding to the short vector $v\in\Lambda\Phi(x_n)g_{t_n}$ are $v\Phi_{n}(x)$.  Player B wishes to choose $B_n$ so that for $x\in B_n$ the neighbors $v\Phi_{n}(x)$ are all eventually expanding (or at least not further contracting) under $g_{t_n}$, where the notion of ``eventually" depends only on $\alpha$,~$\beta$ and $\Phi$, but not on $v$, $n$ or the game play up to this point.  If the height of $v\Phi_{n}(x)$ is growing (in a uniform way for all~$x\in B_n$) under $g_{t_m}$ for~$m\geq n$, 
then the same holds for all other vectors in $Kv\Phi_{n}(x)$.  Moreover, by the lemma we know that for the period of
time that the height of our vector is~$<1$ there cannot be any other vector outside the~$K$-span with height~$<1$. Also note that~$\Phi_n(x)$ is uniformly bounded for~$|x|<\rho(\alpha\beta)^n$, which implies that
all vectors~$w\in \Lambda \Phi(x_n)g_{t_n}\Phi_{n}(x)$ of sufficiently small norm are by the lemma in~$Kv\Phi_n(x)$
and so controlled by the move of Player B.

Recall that under an application of $g_{t}$, $v_1^{\sigma}$ is contracted by $e^{-t}$ and $v_2^{\sigma}$ is expanded by $e^{t}$.  Consider the ratio between the expanding direction
in the factor corresponding to some $\sigma \in S$ of $v\Phi_{n}(x)$ with~$a_\sigma\neq 0$ and the norm of $v^{\sigma}$:
\[
\frac{\left|v_1^{\sigma}a_{\sigma}e^{2t_n}x+v_2^{\sigma}\right|}{\|v^{\sigma}\|} = 
\left|a_{\sigma}(\rho(\alpha\beta)^n)^{-1}x\frac{v_1^\sigma}{\|v^\sigma\|}+\frac{v_2^{\sigma}}{\|v^{\sigma}\|}\right|,
\]
where~$a_\sigma$ denotes the slope of~$\Phi$ at~$\sigma$. Also recall that~$x$ is restricted, at this stage in the 
game, to~$x\in (-\rho(\alpha\beta)^n,\rho(\alpha\beta)^n)$. This shows that it is possible to choose~$x$ such
that this ratio is bounded below by a constant~$\e_{\sigma}>0$ (depending on~$a_\sigma$ and~$\alpha$ only).
In fact, depending on the signs of~$v_1^\sigma$ and~$v_2^\sigma$ we can choose a subinterval $B \subset (-\rho(\alpha\beta)^n,\rho(\alpha\beta)^n)$ of radius $\alpha$ such that the ratio is uniformly bounded from below by~$\alpha$. We may choose $B=(\rho(\alpha\beta)^n(1-2\alpha),\rho(\alpha\beta)^n)$ if the real part of $\frac{v_2^{\sigma}}{v_1^{\sigma}}$ is nonnegative, or choose $B=(-\rho(\alpha\beta)^n,-\rho(\alpha\beta)^n(1-2\alpha))$ if the real part of $\frac{v_2^{\sigma}}{v_1^{\sigma}}$ is negative.  

If we make this choice then it follows that the expanding component of~$(v\Phi_n(x))^\sigma$ is of norm at least~$\e_{\sigma}\|v^\sigma\|$ for all~$x$ in the new ball
and this remains true in the future even if the vector initially is contracted.
Moreover, if we know that the coordinate in the expanding direction is of norm at least ~$\e_{\sigma}\|v^\sigma\|$,
then after time $t=|\log{\e_{\sigma}}|$, the expanding direction of the 
factor $(v\Phi_{n}(x)g_t)^{\sigma}$ will be at least as big as~$\|v^{\sigma}\|$ for all $x\in B$. 

Take $\e = \min_{\sigma} \e_{\sigma}$.  If half of the factors with nonzero~$a_\sigma$  agree in the choice of $B$,  then $B$ is chosen to be the rightmost (resp., leftmost) subinterval of $(-\rho(\alpha\beta)^n,\rho(\alpha\beta)^n)$.  After this Player~$A$
makes his move. Now we again look at all components and repeat the vote among those factors that voted differently the first time. After a uniformly bounded number of repetitions of this
voting procedure, say~$m$ repetitions (during which the game moves on) we have ensured that the norm of the expanding component~$v_1^\sigma$ is at least~$\e e^{-2t_m}\|v^\sigma\|$ for all places~$\sigma$ with~$a_\sigma\neq 0$. 
Let~$k=n+m$ and let~$B_k$ be the ball that is chosen by Player B at the last step, then it follows
that the height of~$v\Phi_k(x)g_t$ is uniformly bounded away from zero for all~$x\in B_k$ and all~$t\geq t_k$. 
Indeed, for all places~$\sigma\in S$ with~$a_\sigma\neq0$ we ensured that the expanding direction
is significant in size, and even if all the remaining directions are contracted by~$g_t$ the height will be bounded away
from~$0$ (depending on~$\e$ and~$\height(v)$).

If during the above procedure (or later in the game) a new vector becomes of height less than one, its height is bounded away from zero by a constant depending on~$m$. We then repeat the procedure with the new vector.  

Notice that the strategy for Player B has not depended on the short vector $v$ but on the direction of $v$. Moreover, this is indeed a winning strategy for Player B. The resulting point $x_{\infty}\in \bigcap A_n$ satisfies:  For all $v \in \Lambda\Phi(x_{\infty})$, either $\height(vg_{t_n}) \geq 1$ for all $n$, or in the first round $n$ with  $\height(vg_{t_n}) < 1$, we have that $\height(v g_{t_{n+m}})$ is increasing for positive $n$, whence $\height(vg_t)$ is bounded away from zero for all~$t\geq t_n$.  Thus, no vector in $\{v\in\Lambda\Phi(x_{\infty})g_{t_n}\setminus\{0\}~|~n\in\N\}$ is ever shorter than a constant depending
on~$\alpha,\beta,\e$ and~$m$, except for perhaps the short vectors which appear in $\Lambda\Phi(x_{\infty})$ initially.

\subsection{A counterexample}\label{counterex} Note that we are crucially using the fact that $a_{\sigma} \neq 0$ for at least half the $\sigma$'s (counting complex places double). Indeed, suppose  $d =r= 3$ and in two of the factors, we have $a_{\sigma_1} =a_{\sigma_2}= 0$. 
In this case, it may happen that the vector~$v$ considered above satisfies that~$v^{\sigma_1}$ and $v^{\sigma_2}$
are contracted eigenvectors. However, in that case Player B will always lose --- no matter of the choices in the
game the height of the vector corresponding to~$v$ will go to zero. 
As we will see below, this behavior becomes even more significant in the unequal weights case. 

\subsection{Proof of Theorem \ref{thm:main2} in a special case}

The proof of Theorem \ref{thm:main2}, i.e.~the weighted result for quadratic extensions  in the linear case follows along the same general lines as above. We explain the strategy at the $n$-th stage following the notation in Theorem~\ref{thm:main2}. In particular, we are now acting by $g(\br)_t$. As before, Player A has chosen a subinterval $A_n = B(x_n, \rho(\alpha\beta)^n)\subseteq B_{n-1}$.  This corresponds to the collection of $\O_K$-modules 
\[
\{\Lambda \Phi(x_n+x)g({\bf r})_{t_n})~|~x\in B(0,\rho(\alpha\beta)^n)\}
\]
where $t_n=\frac{1}{2r}\log\frac{1}{\rho(\alpha\beta)^n}$ and  $r=\max_{\sigma\in S}r_{\sigma}$. 

We assume that there exists nonzero $v\in\Lambda\Phi(x_n)g({\br})_{t_n}$ with $\height(v)<1$. 
Let~$\sigma_0\in S$ be a place with~$r_{\sigma_0}=r$. By Lemma \ref{lem:onespan}, we again have to worry about at most one direction.  Define $\Phi_{n}(x)=g({\bf r})_{t_n}^{-1}\Phi(x)g({\bf r})_{t_n}$ so that $\Lambda\Phi(x_n+x)g({\bf r})_{t_n} = \Lambda \Phi(x_n)g({\bf r})_{t_n}\Phi_{n}(x)$, denote by $v$ the short vector and consider $v\Phi_{n}(x)$.  Under an application of $g({\bf r})_{t}$, $v_1^{\sigma_0}$ is contracted by $e^{-rt}$ and $v_2^{\sigma_0}$ is expanded by $e^{rt}$. In the component corresponding to~$\sigma_0$ consider the ratio between the expanding coordinate $v\Phi_{n}(x)$ and the norm of the vector:
\[
\left|\frac{v_1^{\sigma_0}a_{\sigma_0}e^{2t_nr}x+v_2^{\sigma_0}}{\|v^{\sigma_0}\|}\right| 
= \left|a_{\sigma_0}(\rho(\alpha\beta)^n)^{-1}x\frac{v_1^{\sigma_0}}{\|v^{\sigma_0}\|}+\frac{v_2^{\sigma_0}}{\|v^{\sigma_0}\|}\right|
\]
 \noindent This ratio can be guaranteed to be larger than $|a_{\sigma_0}|(1-2\alpha)$ for all $x$ in a subinterval $B \subset (-\rho(\alpha\beta)^n,\rho(\alpha\beta)^n)$ of radius $\alpha$, choosing $B=(\rho(\alpha\beta)^n(1-2\alpha),\rho(\alpha\beta)^n)$ if $\frac{v_2^{\sigma_0}}{v_1^{\sigma_0}}$ is nonnegative, or choosing $B=(-\rho(\alpha\beta)^n,-\rho(\alpha\beta)^n(1-2\alpha))$ if $\frac{v_2^{\sigma_0}}{v_1^{\sigma_0}}$ is negative.
(If~$v_1^{\sigma_0}=0$ then the vector~$v^{\sigma_0}$ is an expanding eigenvector already anyway and the ratio is one.)  

We now argue as before. As the modified vector~$v\Phi_n(x)$ for~$x$ in the new subinterval~$B$
has a significant expanding component in the place~$\sigma_0$ and since this place is the one with the faster dynamics, it follows that the height of~$v\Phi_n(x)g(\br)_t$ will be~$\geq 1$ for~$x\in B$ and 
for~$t\geq t_0=t_0(a_{\sigma_0},\alpha,\br)$. As before, since~$\Phi_n(x)$ is bounded no vector in~$\Lambda\Phi(x_n+x)g({\br})_{t_n}=\Lambda\Phi(x_n)g({\br})_{t_n}\Phi_n(x)$ will have height much smaller than~$v$. 
Once more we obtain a winning strategy for Player B.


 In this case, once more the lower bound on derivatives which amounts to $a_{\sigma} \neq 0$ for \emph{both} $\sigma$ is crucial. If not, we may consider a similar example to the one in \S \ref{counterex}, i.e.\ suppose $a_{\sigma'} = 0$ for one factor.
If now~$\sigma'=\sigma_0$ for some weight~$\br$, then the above strategy fails. Moreover, in that case it may happen that our lattice contains a vector~$v=(v^{\sigma_0},v^{\sigma_1})$ with~$v^{\sigma_0}$ being contracted by~$e^{r_{\sigma_0}t}$ under the dynamics of~$g(\br)_t$. Even if~$v^{\sigma_1}$ is now expanded by~$e^{r_{\sigma_1}t}$, then height of
the vector~$v g(\br)_t$ will still go to zero and this remains true for every possible outcome of the game. Hence in
this case there cannot exist a winning strategy.

\section{Proof in the general case}
We now discuss the proof of Theorem \ref{thm:main1} in the general case of $\phi$ in $C^1$. Recall that we have assumed that we have a subset $S' \subset S$ with $|\{\sigma \in S'~|~\sigma \text{ is real} \}| + 2|\{\sigma \in S'~|~\sigma \text{ is complex} \}| > \lfloor \frac{d}{2} \rfloor$ and for each $\sigma \in S'$, only finitely many points $x$ have $\phi_{\sigma}'(x)=0$.  Here we use linear approximations to $\phi$, and the strategy explained above for linear functions.  At the beginning of the game, Player B acts by moving the playing field away from all points $x$ with $\phi_{\sigma}'(x)=0$ for any $\sigma\in S'$. Since this is assumed to be a finite set, this takes only finitely many rounds of the game.  Thus, the game arrives at round $N$ in the situation that for all $\sigma \in S'$, $\phi_{\sigma}'(x)\neq 0$ for all $x\in \overline{B_N}$.  Since $\phi_{\sigma}'$ is  uniformly  continuous on $\overline{B_N}$, there exists for each $\sigma\in S'$, $m_{\sigma},M_{\sigma}\in \R$ such that $0\not\in [m_{\sigma},M_{\sigma}]$ while for all $x\in B_N, \phi_{\sigma}'(x)\in[m_{\sigma},M_{\sigma}]$.  Player B will use these bounds to produce the piecewise linear approximation to $\phi$.  

Suppose that in round $n>N$ there is a nonzero vector $v\in\Lambda\Phi(x_n)$ with $\height(vg_{t_n})<1$.  (Again we choose $t_n = \frac1 2 \log \rho(\alpha \beta)^{-n}$.) Define
\[
\hat{\Phi}_{n}(x) = g^{-1}_{t_n}\left(\begin{pmatrix}1& \phi_{\sigma}(x_n+x)-\phi_{\sigma}(x_n) \\ 0&1\end{pmatrix}\right)_{\sigma}g_{t_n}.
\]
Then as before Player B would like to choose $B_n$ so that  the ``angle" between  $v\hat{\Phi}_{n}(x)$ and the contracting direction is significant.  Thus for $v\hat{\Phi}_{n}(x)\in\Lambda\Phi(x_n+x)g_{t_n}$ and $\sigma\in S'$, Player B considers the ratio
\[
\left| \frac{v_1^{\sigma}(\phi_{\sigma}(x_n+x)-\phi_{\sigma}(x_n))e^{2t_n}+v_2^{\sigma}}{\|v^{\sigma}\|} \right| 
= \left|(\phi_{\sigma}(x_n+x)-\phi_{\sigma}(x_n))(\rho(\alpha\beta)^n)^{-1} \frac{v_1^{\sigma}}{\|v^{\sigma}\|} + \frac{v_2^{\sigma}}{\|v^{\sigma}\|}\right|,
\]
and wishes to bound this quantity from below uniformly over $x$ in the yet to be determined $B_n$.
Since $\phi$ is monotone, let us first suppose that $\phi$ is increasing, that is, we assume $0<m_{\sigma}\leq M_{\sigma}$.  Then on $[0,\rho(\alpha\beta)^n)$  we have that $\phi_{\sigma}(x_n+x)-\phi_{\sigma}(x_n)\geq m_{\sigma}x$ and on $(-\rho(\alpha\beta)^n,0)$, $\phi_{\sigma}(x_n+x)-\phi_{\sigma}(x_n)\geq M_{\sigma}x$.  Using these linear approximations of $\phi$, Player B uses the same strategy as before.  If the real part of $ \frac{v_2^{\sigma}}{v_1^{\sigma}}$ is nonnegative, then for $x$ in $((1-2\alpha)\rho(\alpha\beta)^n,\rho(\alpha\beta)^n)$,
\[
 \left|(\phi_{\sigma}(x_n+x)-\phi_{\sigma}(x_n))(\rho(\alpha\beta)^n)^{-1}\frac{v_1^{\sigma}}{\|v^{\sigma}\|} + \frac{v_2^{\sigma}}{\|v^{\sigma}\|}\right| \geq m_{\sigma}(1-2\alpha).
\]
Similarly, if the real part of $ \frac{v_2^{\sigma}}{v_1^{\sigma}}$ is negative, then for $x$ in $(-\rho(\alpha\beta)^n,-(1-2\alpha)\rho(\alpha\beta)^n)$,
\[
 \left|(\phi_{\sigma}(x_n+x)-\phi_{\sigma}(x_n))(\rho(\alpha\beta)^n)^{-1}\frac{v_1^{\sigma}}{\|v^{\sigma}\|} + \frac{v_2^{\sigma}}{\|v^{\sigma}\|}\right| \geq M_{\sigma}(1-2\alpha).
\]
Player B uses a similar analysis in the case that $\phi_{\sigma}$ is decreasing.  So for fixed $\sigma\in S' $, Player B can choose $B_n$ so that for all $x\in B_n$ the ratio between the coordinates in the factor corresponding to $\sigma$ of any neighbor $v\hat{\Phi}_{n}(x)$ is greater than $\inf_{y\in B_N}|\phi_{\sigma}'(y)|(1-2\alpha)$.  

As before, Player B should execute this strategy over several, say $m$, rounds of the game to ensure that $\height(v\hat{\Phi}_{n}(x)g_{t})$ is increasing for all $t>t_m + |\log \e|$, where $\e = \linebreak \min_{\sigma\in S' } \inf_{y\in \overline{B_N}}|\phi_{\sigma}'(y)|(1-2\alpha)$.  The same reasoning as in the special case shows that this strategy is winning.

\noindent The same strategy works for the more general case of Theorem \ref{thm:main2}.

\end{document}